\documentclass[12pt]{amsart}

\usepackage{amsmath}
\usepackage{amsfonts}
\usepackage{amssymb}
\usepackage{amscd}
\usepackage{graphicx}
\usepackage[abbrev,alphabetic]{amsrefs}
\RequirePackage[dvipsnames,usenames]{xcolor}
\usepackage{soul,xcolor}
\setstcolor{red}
\usepackage{stmaryrd}
\usepackage{mathtools}
\usepackage{booktabs}
\usepackage{multirow}
\newtagform{tiny}{\tiny(}{)}
\usepackage{blkarray}
\usepackage{mathrsfs}

\usepackage{mathtools}
\usepackage{hyperref}
\usepackage[margin=1.25in]{geometry}

\usepackage{amsthm}
\usepackage{comment}
\usepackage[all,cmtip]{xy}
\usepackage{tikz-cd}
\usetikzlibrary{cd}

\newtheorem{thm}{Theorem}[section]
\newtheorem{lem}[thm]{Lemma}

\newtheorem{claim}[thm]{Claim}
\newtheorem{prop}[thm]{Proposition}

\theoremstyle{definition}

\newtheorem{question}[thm]{Question}

\newtheorem{defn}[thm]{Definition}
\newtheorem{rmk}[thm]{Remark}
\newtheorem*{ack}{Acknowledgements}
\numberwithin{equation}{section}
\newcommand{\var}{\overline}

\theoremstyle{plain}


\def\R{{\mathbb R}}
\def\Z{{\mathbb Z}}

\def\var{\overline}

\def\sO{\mathcal{O}}

\DeclareMathOperator{\Pic}{Pic}

\DeclareMathOperator{\Spec}{Spec}

\def\X{{\mathcal X}}

\AtBeginDocument{%
  \def\MR#1{}
}

\newenvironment{claimproof}[0]
  {%
   \paragraph{\it Proof.}%
  }
  {%
    \hfill$\blacksquare$%
  }

\makeatletter
\def\@tocline#1#2#3#4#5#6#7{\relax
  \ifnum #1>\c@tocdepth 
  \else
    \par \addpenalty\@secpenalty\addvspace{#2}%
    \begingroup \hyphenpenalty\@M
    \@ifempty{#4}{%
      \@tempdima\csname r@tocindent\number#1\endcsname\relax
    }{%
      \@tempdima#4\relax
    }%
    \parindent\z@ \leftskip#3\relax \advance\leftskip\@tempdima\relax
    \rightskip\@pnumwidth plus4em \parfillskip-\@pnumwidth
    #5\leavevmode\hskip-\@tempdima
      \ifcase #1
       \or\or \hskip 1em \or \hskip 2em \else \hskip 3em \fi%
      #6\nobreak\relax
    \hfill\hbox to\@pnumwidth{\@tocpagenum{#7}}\par
    \nobreak
    \endgroup
  \fi}
\makeatother

\title[Varieties  with numerically flat log cotangent bundle]
{Varieties in positive characteristic with numerically flat log cotangent bundle}
\author{Sho Ejiri, Shou Yoshikawa}
\address{Department of Mathematics, Graduate School of Science, Osaka Metropolitan University, Osaka City, Osaka 558-8585, Japan}
\email{shoejiri.math@gmail.com}
\address{Tokyo Institute of Technology, 2-12-1 Ookayama, Meguro-ku, Tokyo 152-8550 Japan}
\email{yoshikawa.s.al@m.titech.ac.jp}
\begin{document}

\begin{abstract}
In this paper, we prove that a smooth projective globally $F$-split variety with numerically flat tangent bundle is an \'etale quotient of an ordinary abelian variety.
We also show its logarithmic analog, which contains a characterization of toric varieties.
We further prove that, without assumption of global $F$-splitting, a smooth projective separably rationally connected 
variety of arbitrary characteristic with numerically flat log cotangent bundle is a toric variety. 
\end{abstract}

\maketitle

\setcounter{tocdepth}{2}


\section{Introduction}

\subsection{Positivity of tangent bundles}
The positivity condition imposed on the tangent bundle of a smooth projective variety is known to restrict the geometric structure of the variety. 
Hartshorne~\cite{Har70} conjectured that if the tangent bundle is ample, where ampleness is a strong positivity condition defined for vector bundles, then the variety is isomorphic to a projective space.  
This conjecture was verified by Mori~\cite{Mor79}.
Furthermore, Demailly--Peternell--Schneider~\cite{DPS94} established a decomposition theorem for a smooth projective variety over $\mathbb C$ with nef tangent bundle, where  nefness is roughly defined as the ``limit'' of ampleness.
The theorem states that, up to an \'etale cover, such a variety has a smooth fibration
over an abelian variety whose fibers are Fano varieties, so one can say that such a variety decomposes into the ``positive'' part and the ``flat'' part.

A positive characteristic analog of the above decomposition theorem was proved by Kanemitsu and Watanabe~\cite{KW20}.
The ``flat'' part of their theorem is a smooth projective variety with numerically flat tangent bundle, and they left the following question:
\begin{question}\label{ques:abelian}
Is a smooth projective variety with numerically flat tangent bundle an \'etale quotient of an abelian variety?
\end{question}
\noindent
Here, a vector bundle $E$ is said to be numerically flat if both $E$ and $E^\vee$ are nef. 
Note that, in characteristic zero, Question~\ref{ques:abelian} is solved affirmatively by using the Beauville--Bogomolov decomposition.
When the tangent bundle is free and the variety is globally $F$-split, 
a theorem of Mehta and Srinivas~\cite{metha-srinivas} answers Question~\ref{ques:abelian} affirmatively.
We say that a variety $X$ is \textit{globally $F$-split} if the Frobenius morphism $F^\sharp:\mathcal O_X \to F_*\mathcal O_X$ splits as an $\mathcal O_X$-homomorphism.
In the case when the variety is \textbf{not} globally $F$-split, Question~\ref{ques:abelian} is open even if the tangent bundle is free.

In this paper, we answer  Question~\ref{ques:abelian} affirmatively when the variety is globally $F$-split.
\begin{thm} \label{intro:thm:characterization of abelian varieties} \samepage
Let $X$ be a smooth projective  variety over an algebraically closed field of positive characteristic.
Then the following are equivalent:
\begin{enumerate}
\item $X$ is an \'etale quotient of an ordinary abelian variety;
\item $X$ is globally $F$-split and the tangent bundle $T_X$ is numerically flat.
\end{enumerate}
\end{thm}
\noindent
We say that an abelian variety $A$ of positive characteristic is \textit{ordinary} if the $p$-rank is equal to the dimension, and it is known that the ordinarity of $A$ is equivalent to the global $F$-splitting of $A$.
The numerical flatness of $T_X$ is known to be equivalent to the condition that there exists an ample divisor $H$ on $X$ such that $T_X$ is $H$-semistable and
\[
\mathrm{ch}_1(X)\cdot H^{n-1}=\mathrm{ch}_2(X) \cdot H^{n-2}=0
\]
(see Proposition \ref{prop:numerical flat}), where $n$ is the dimension of $X$.
Theorem~\ref{intro:thm:characterization of abelian varieties} is a corollary of the main theorem of this paper (Theorem~\ref{intro:thm:toric fibration}) mentioned in the next subsection.

Combining Theorem~\ref{intro:thm:characterization of abelian varieties} with Kanemitsu and Watanabe's theorem, we obtain the following decomposition theorem:
\begin{thm} \label{intro:thm:decomposition} \samepage
Let $X$ be a smooth projective variety over an algebraically closed field of positive characteristic. 
Suppose that the tangent bundle $T_X$ is nef and $X$ is globally $F$-split.
Then there exists a finite \'etale cover $f:Y\to X$ and a smooth algebraic fiber space $\varphi:Y\to A$ such that 
\begin{enumerate}
\item $\varphi$ is the MRCC fibration of $Y$, 
\item every fiber of $\varphi$ is a globally $F$-split and separably rationally connected Fano variety with nef tangent bundle, and
\item $A$ is an ordinary abelian variety. 
\end{enumerate}
\end{thm}
\subsection{Numerically flat logarithmic cotangent bundles}
Let $X$ be a smooth projective variety over an algebraically closed field
and let $D$ be a normal crossing divisor on $X$.
Similarly to the case of $D=0$, some geometry of the pair $(X,D)$ are derived from the property of the sheaf $\Omega_X(\mathrm{log}\,D)$ of differentials with log poles along $D$.
In characteristic zero, Winkelmann~\cite{Win04} proved that if  $X$ is rationally connected and $\Omega_X(\mathrm{log}\,D)$ is free, then $(X,D)$ is a toric pair.
%
In positive characteristic, Achinger--Witaszek--Zdanowicz~\cite{AWZ21} showed that the following are equivalent:
\begin{enumerate}
\item[(a)] there exists a finite \'etale cover $f:Y\to X$ such that $Y$ has a toric fibration (see Definition~\ref{defn:toric fibration}) over an ordinary abelian variety with toric boundary $f^*D$;
\item[(b)] $X$ is globally $F$-split and $\Omega_X(\mathrm{log}\,D)$ becomes free on a finite \'etale cover of $X$.
\end{enumerate}

Condition~(b) implies that $\Omega_X(\mathrm{log}\,D)$ is numerically flat,
so it is natural to ask ``can condition~(b) be weakened to the condition that $\Omega_X(\mathrm{log}\,D)$ is numerically flat?''
This question is solved affirmatively by the following theorem, which is the main theorem of this paper.
\begin{thm} \label{intro:thm:toric fibration} \samepage
Let $X$ be a smooth projective variety over an algebraically closed field of positive characteristic.
Let $D$ be a normal crossing divisor on $X$.
Then the following are equivalent:
\begin{enumerate}
\item there exists a finite \'etale cover $f:Y\to X$ such that $Y$ has a toric fibration over an ordinary abelian variety with toric boundary $f^*D$; 
\item $X$ is globally $F$-split and $\Omega_X(\mathrm{log}\,D)$ is numerically flat. 
\end{enumerate}
\end{thm}
Although Theorems~\ref{intro:thm:characterization of abelian varieties}, \ref{intro:thm:decomposition} and~\ref{intro:thm:toric fibration} need the assumption that $X$ is globally $F$-split, we can get rid of the assumption if $X$ is separably rationally connected.
This is a part of the following theorem that is a characterization of toric varieties in positive characteristic.
\begin{thm} \label{intro:thm:toric_p} \samepage
Let $X$ be a smooth projective variety over an algebraically closed field of positive characteristic.
Let $D$ be a normal crossing divisor on $X$.
Then the following are equivalent:
\begin{enumerate}
\item $X$ is a toric variety with toric boundary $D$;
\item $\Omega_X(\mathrm{log}\,D)$ is numerically flat and $X$ is separably rationally connected;
\item $\Omega_X(\mathrm{log}\,D)$ is numerically flat and $X$ is rationally connected and globally $F$-split.
\end{enumerate}
\end{thm}
\noindent
As an application of this theorem, we give a characterization of toric varieties in characteristic zero:
\begin{thm} \label{intro:thm:toric_0} \samepage
Let $X$ be a smooth projective variety over an algebraically closed field of characteristic zero.
Let $D$ be a normal crossing divisor on $X$.
Then the following are equivalent:
\begin{enumerate}
\item $X$ is a toric variety with toric boundary $D$; 
\item $\Omega_X(\mathrm{log}\,D)$ is numerically flat and $X$ is rationally connected.
\end{enumerate}
\end{thm}
\noindent
This theorem also follows from a result of Druel and Lo Bianco~\cite{DL22},
but their proof contains an analytic method, so it is completely different from our proof that is purely algebraic.


\begin{ack}
The authors would like to thank the organizers of ``OCAMI Arithmetic and Dynamics Seminar'' where this collaboration started.
They are grateful to Shunsuke Takagi, Kenta Sato, Tasturo Kawakami and Teppei Takamatsu for helpful comments.
The second author was supported by JSPS KAKENHI Grant number JP20J11886 and RIKEN iTHEMS Program.
\end{ack}

\section{Preliminary}

\subsection{Numerical flatness}
In this subsection, we give a characterization of numerically flatness of vector bundles.

\begin{defn}\label{defn:numerically flat}
Let $X$ be a projective variety over a field and $E$ a vector bundle on $X$.
We say that $E$ is \emph{numerically flat} if both $E$ and $E^{\vee}$ are nef.
\end{defn}

\begin{prop}\label{prop:numerical flat}
Let $X$ be a smooth $n$-dimensional projective variety over a perfect field of positive characteristic and $D$ be a normal crossing pair on $X$.
Then $\Omega_X(\log\,D)$ is numerically flat if and only if
there exists an ample Cartier divisor $H$ on $X$ such that $\Omega_X(\log\,D)$ is $H$-semistable and
\[
\mathrm{ch}_1(\Omega_X(\log\,D))\cdot H^{n-1}=\mathrm{ch}_2(\Omega_X(\log\,D)) \cdot H^{n-2}=0.
\]
\end{prop}

\begin{proof}
The ``only if'' part follows from \cite{Lan12}*{Theorem~2.2}.
Let us show the ``if'' part.
By assumption, we have
$\mu(\Omega_X(\log\,D))=\mu_{max}(\Omega_X(\log\,D))=0$.
Therefore, we obtain $\mu_{max}(\Omega_X) \leq \mu_{max}(\Omega_X(\log\,D)) \leq 0$ by the inclusion $\Omega_X \subseteq \Omega_X(\log\,D)$.
By what mentioned in \cite{Lan04}*{p.\,275}, $\Omega_X(\log\,D)$ is strongly $H$-semistable.
Thus, by \cite{Lan12}*{Theorem 2.2}, $\Omega_X(\log\,D)$ is numerically flat.
\end{proof}


\subsection{Cartier operators on families}
In this subsection, we work over a perfect field $k$ of characteristic $p>0$.
In order to reduce Theorem \ref{intro:thm:toric fibration} to the case when the base field is a finite field, we study the Cartier operators on families.

\begin{defn}
Let $X$ be a variety.
We say that $X$ is \emph{globally $F$-split} if the natural homomorphism
\[
F^{\#} \colon \sO_X \to F_*\sO_X
\]
splits as an $\sO_X$-module homomorphism.
\end{defn}

\begin{defn}\label{definition: relative Frobenius log version}\textup{(cf.~\cite{AWZ21}*{Section 2.3})}
Let $\mu \colon \X \to S$ be a morphism of varieties and $(\X,D)$ a normal crossing pair over $S$.
We consider the following diagram;
\[
\xymatrix{
\X \ar[rd] \ar@/^15pt/[rrd]^-\mu \ar@/_15pt/[rdd]_-F & & \\
& \X' \ar[r]^-{\mu'} \ar[d]_i \ar@{}[rd]|\Box & S \ar[d]^-F \\
& \X \ar[r]_-\mu & S.
}
\]
Then the morphism $\X \to \X'$ is denoted by $F_{\X/S}$ and called the \emph{relative Frobenius} of $\X$ over $S$.
We set 
\begin{eqnarray*}
    Z^1_{\X/S}(\mathrm{log} D)&:=&\mathrm{Ker}((F_{\X/S})_*\Omega_{\X/S}(\mathrm{log} D) \to (F_{\X/S})_*\Omega^2_{\X/S}(\mathrm{log} D)), \\
    B^1_{\X/S}&:=&\mathrm{Im}((F_{\X/S})_*\sO_{\X} \to (F_{\X/S})_*\Omega_{\X/S}).
\end{eqnarray*}
We note that they are coherent sheaves on $\X'$.
\end{defn}

\begin{prop}\label{prop: relative Frobenius log version}
We use the setting and notation in Definition \ref{definition: relative Frobenius log version}.
\begin{enumerate}
    \item If $\sO_{\X'} \to (F_{\X/S})_*\sO_{\X}$ splits, then  for every geometric point $\overline{s}$ of $S$, the fiber $\X_{\overline{s}}$ is globally $F$-split.
    \item If $\mu$ is proper, then the set
    \[
    \{ s \in S \ |\ \X_{\bar{s}} \textup{ is globally $F$-split}\}
    \]
    is constructible, where $\bar{s}$ is the geometric point associated to $s$.
    \item  We have the following exact sequence
    \begin{equation}\label{eq:relative obst}
    \xymatrix{
    0 \ar[r] & B^1_{\X/S} \ar[r] & Z^1_{\X/S}(\mathrm{log} D) \ar[r]^-{C_{\X/S}} & i^*\Omega_{\X/S}(\mathrm{log} D) \ar[r] & 0
    }
    \end{equation}
    by shrinking $S$.
    \item The restriction of $(\ref{eq:relative obst})$ to a geometric fiber $\X_{\overline{s}}$ is isomorphic to the exact sequence
    \[
    \xymatrix{
    0 \ar[r] & B^1_{\X_{\overline{s}}} \ar[r] & Z^1_{\X_{\overline{s}}}(\mathrm{log} D) \ar[r]^-{C_{\X_{\overline{s}}}} & \Omega_{\X_{\overline{s}}}(\mathrm{log} D) \ar[r] & 0
    }
    \]
    by shrinking $S$.
\end{enumerate}

\end{prop}

\begin{proof}
Let $\overline{s}$ be a geometric point of $S$.
We have a following diagram;
\[
\xymatrix{
\X_{\overline{s}} \ar[rd]^-{F_{\X/S, \overline{s}}} \ar@/^15pt/[rrd] \ar@/_15pt/[rdd]_-F & & \\
& \X'_{\overline{s}} \ar[r] \ar[d]_{i_{\overline{s}}} \ar@{}[rd]|\Box & \mathrm{Spec} (\kappa(\overline{s})) \ar[d]^-F \\
& \X_{\overline{s}} \ar[r] & \mathrm{Spec} (\kappa(\overline{s})).
}
\]
Since $\kappa(\overline{s})$ is a perfect field, $i_{\overline{s}}$ is an isomorphism.
In particular, the homomorphism $F^{\#}$ is a composition of an isomorphism and a homomorphism $F_{\X/S, \overline{s}}^{\#}$, thus we obtain assertion $(1)$.
Next, we consider the exact sequence
\[
0 \longrightarrow \sO_{\X'} \longrightarrow F_{\X/S *}\sO_{\X} \longrightarrow B^1_{\X/S} \longrightarrow 0.
\]
Since $\mu$ is smooth, $B^1_{\X/S}$ is locally free.
Thus, the extension class of above exact sequence is corresponding to an element
\[
\alpha \in H^1(\X,(B^1_{\X/S})^{\vee}) \simeq \mathrm{Ext}^1(B^1_{\X/S},\sO_{\X}).
\]
By the argument in the proof of $(1)$, $\X_{\bar{s}}$ is globally $F$-split if and only if the image $\alpha_{\bar{s}}$ in $H^1(\X_{\bar{s}},(B^1_{\X_{\bar{s}}})^{\vee})$ is zero.
Therefore, the set in assertion $(2)$ is constructible.
By shrinking $S$, for every geometric point $\overline{s}$ of $S$, we have $B^1_{\X/S, \overline{s}} \simeq i_{\overline{s}}^* B^1_{\X_{\overline{s}}}$, $Z^1_{\X/S, \overline{s}}(\mathrm{log} D) \simeq i_{\overline{s}}^* Z^1_{\X_{\overline{s}}}(\mathrm{log} D_{\overline{s}})$, and $\Omega^1_{\X/S, \overline{s}}(\mathrm{log} D) \simeq \Omega^1_{\overline{s}}(\mathrm{log} D_{\overline{s}})$.
Therefore, we can define the relative Cartier operator $C_{\X/S}$ and we obtain assertions (3) and (4).
\end{proof}

\section{Toric fibrations over ordinary abelian varieties}
In this section, we prove Theorem \ref{intro:thm:toric fibration}.
First, we study the splitting of the exact sequence
\[
0 \longrightarrow B^1_X \longrightarrow Z^1_X \longrightarrow \Omega_X \longrightarrow 0
\]
over a finite field.

\begin{defn}\label{defn:toric fibration}\textup{(cf.~\cite{AWZ21}*{Definition 2.1.1, Lemma 2.1.2},)}
Let $S$ be a scheme.
A \emph{toric fibration} over $S$ is a flat $S$-scheme $X$ together with an action of a torus $T$ over $S$ such that \'etale-locally on $S$, there exists isomorphisms $T \simeq \mathbb{G}^n_{m,S}$ and $X \simeq X(\Sigma)_S$ for some rational polyhedral fan $\Sigma \subseteq \R^n$.
Furthermore, a \emph{toric boundary} of a toric fibration $X \to S$ is defined by glueing toric boundaries $D(\Sigma)_S$.
\end{defn}

\begin{lem}\label{lem:pre chara}\textup{(cf.~\cite{AWZ21}*{Theorem 5.1.1})}
Let $(X,D)$ be a normal crossing pair over a perfect field of positive characteristic.
Then the following are equivalent:
\begin{enumerate}
    \item $X$ admits a finite  \'etale cover $\pi \colon Y \to X$ such that $Y$ has a toric fibration over an ordinary abelian variety with toric boundary $f^*D$;
    \item the exact sequence
    \begin{align*}
    \xymatrix{
    0 \ar[r] &
    B^1_X \ar[r] &
    Z^1_X(\mathrm{log}\,D) \ar[r] &
    \Omega_X(\mathrm{log}\,D) \ar[r] &
    0
    }
    \end{align*}
    splits.
\end{enumerate}
\end{lem}

\begin{proof}
Taking a base change, we may assume the base field is an algebraically closed field.
Then the equivalence follows from the proof of \cite{AWZ21}*{Theorem 5.1.1}.
\end{proof}

\begin{lem} \label{lem:kill}
Let the base field be a finite field.
Let $E$ be a numerically flat vector bundle on a globally $F$-split smooth projective variety $X$. 
If we take $\xi \in H^1(X,E)$, then there exists an \'etale cover $\pi:Y\to X$ such that $\pi^*\xi=0$ 
in $H^1(Y,\pi^*E)$. 
\end{lem}
\begin{proof}
Since $H^1(X,E)\cong \mathrm{Ext}^1(\mathcal O_X, E)$, 
the element $\xi$ corresponds to the exact sequence 
\begin{align*}
\tag{$\xi$}
\xymatrix{
0 \ar[r] & 
E \ar[r] &
G \ar[r] &
\mathcal O_X \ar[r] &
0. 
}
\end{align*}
Note that $G$ is a numerically flat vector bundle.
By \cite{PZ19}*{Lemma~2.5}, there is an \'etale cover $\pi:Y\to X$ such that
${F^e}^*\pi^*E$ and ${F^e}^*\pi^*G$ are free for some $e\ge 1$.
Then the induced exact sequence
\begin{align*}
\tag{${F^e}^*\pi^*\xi$}
\xymatrix{
0 \ar[r] & 
{F^e}^*\pi^*E \ar[r] &
{F^e}^*\pi^*G \ar[r] &
\mathcal O_Y \ar[r] &
0 
}
\end{align*}
splits, i.e., ${F^e}^*\pi^*\xi=0$.
Since $X$ is $F$-split, so is $Y$, and hence $\pi^*\xi=0$.
Note that $\pi^*\xi\in\mathrm{Ext}^1(\mathcal O_Y, \pi^*E)\cong H^1(Y,\pi^*E)$.
\end{proof}
\begin{thm}\label{thm:chara toric fib finite feld}
Let the base field be a finite field.
Let $X$ be a smooth projective variety
and let $D$ be a normal crossing divisor on $X$.
Suppose that $X$ is globally $F$-split and
$\Omega_X(\mathrm{log}\,D)$ is numerically flat.
Then the exact sequence
\begin{align*}
\tag{$\xi$}
\xymatrix{
0 \ar[r] &
B^1_X \ar[r] &
Z^1_X(\mathrm{log}\,D) \ar[r] &
\Omega_X(\mathrm{log}\,D) \ar[r] &
0
}
\end{align*}
splits.
\end{thm}
\begin{proof}
The exact sequence ($\xi$) corresponds to an element
$$
\xi \in \mathrm{Ext}^1(\Omega_X(\mathrm{log}\,D),B_X^1)
\cong H^1(B_X^1\otimes\Omega_X(\mathrm{log}\,D)^\vee).
$$
Since $X$ is globally $F$-split, $F_*\sO_X \to B^1_X$ is a splitting surjection,
so there is
\begin{align*}
\xi'\in H^1(X, (F_*\mathcal O_X) \otimes \Omega_X(\mathrm{log}\,D)^\vee)
& \cong H^1\big(X, F_*(F^*\Omega_X(\mathrm{log}\,D)^\vee)\big)
\\ & \cong H^1(X, F^*\Omega_X(\mathrm{log}\,D)^\vee)
\end{align*}
such that $\alpha(\xi')=\xi$, where
$$
\alpha:H^1(X, F^*\Omega_X(\mathrm{log}\,D)^\vee)
\to H^1(X, B_X^1 \otimes \Omega_X(\mathrm{log}\,D)^\vee)
$$
is the induced map from the splitting surjection $F_*\mathcal O_X \to B_X^1$.
Since $F^*\Omega_X(\mathrm{log}\,D)$ is numerically flat,
by Lemma~\ref{lem:kill}, there is an \'etale cover $\pi:Y\to X$ such that
$\pi^*\xi'=0$.
As $F_*\pi^* \cong \pi^*F_*$, we have the commutative diagram
$$
\xymatrix{
H^1(X, F^*\Omega_X(\mathrm{log}\,D)^\vee) \ar[r]^-{\pi^*} \ar[d]_-\cong
\ar@/_90pt/[dd]_-\alpha &
H^1(Y, \pi^*F^*\Omega_X(\mathrm{log}\,D)^\vee) \ar[d]^-\cong
\ar@/^90pt/[dd]^-\beta \\
H^1(X, F_*(F^*\Omega_X(\mathrm{log}\,D)^\vee)) \ar[r]^-{\pi^*} \ar[d] &
H^1\big(Y, \pi^*F_*(F^*\Omega_X(\mathrm{log}\,D)^\vee)\big) \ar[d] \\
H^1(X, B_X^1\otimes \Omega_X(\mathrm{log}\,D)^\vee) \ar[r]^-{\pi^*} &
H^1\big(Y, \pi^*(B_X^1 \otimes\Omega_X(\mathrm{log}\,D)) \big),
}
$$
where $\beta$ is the induced morphism.
Hence, $\pi^*\xi=\pi^*\alpha(\xi')=\beta(\pi^*\xi') =\beta(0) =0$,
which means that the exact sequence
\begin{align*}
\tag{$\pi^*\xi$}
\xymatrix{
0 \ar[r] &
B^1_Y \ar[r] &
Z^1_Y(\mathrm{log}\,\pi^*D) \ar[r] &
\Omega_Y(\mathrm{log}\,\pi^*D) \ar[r] &
0
}
\end{align*}
splits.
By Lemma \ref{lem:pre chara}, $\xi$ also splits.
\end{proof}

\begin{thm}\label{thm: chara of toric fib}
Let $X$ be a smooth projective variety over an algebraically closed field of characteristic $p>0$
and let $D$ be a normal crossing divisor on $X$.
Suppose that $X$ is globally $F$-split and
$\Omega_X(\mathrm{log}\,D)$ is numerically flat.
Then $X$ admits a finite \'etale cover $\pi \colon Y \to X$ such that  $Y$ has a toric fibration over an ordinary Abelian variety with toric boundary $f^*D$. 
\end{thm}

\begin{proof}
There exists a smooth affine variety $S$ over a finite field, a smooth projective morphism $\mu \colon \mathcal{X} \to S$, and a reduced divisor $\mathcal{D}$ on $\X$ such that
\begin{enumerate}
    \item $(\X,\mathcal{D})$ is normal crossing over $S$,
    \item $\kappa(\eta) \subset k$, where $\eta \in S$ is the generic point,
    \item $\mathcal{X} \times_S \mathrm{Spec}(k) \simeq X$,
    \item for every closed point $s \in S$,  $X_s$ is globally $F$-split, 
    \item for every closed point $s \in S$, $\Omega_{X_s}(\log\,D_s)$ is numerically flat, and
    \item for every closed point $s \in S$, the restriction map
    \[
    H^1(\X',i^*\Omega_{\X/S}(\log\,\mathcal{D})^{\vee} \otimes B_{\X/S})\otimes \kappa(s) \to H^1(\X_s,\Omega_{X_s}(\log\,D_s)^\vee \otimes B^1_{\X_s})
    \]
    is an isomorphism,
\end{enumerate}
where we use the notation in Definition \ref{definition: relative Frobenius log version}, (4) follows from Proposition \ref{prop: relative Frobenius log version} (2), and (5) follows from Proposition \ref{prop:numerical flat} and the openness of semistability \cite{HL10}*{Proposition 3.1}.
We denote the extension class of the exact sequence (\ref{eq:relative obst}) in Proposition \ref{prop: relative Frobenius log version} by $\xi_{\X} \in H^1(\X',i^*\Omega_{\X/S}(\log\,\mathcal{D})^{\vee} \otimes B_{\X/S})$.
By Proposition \ref{prop: relative Frobenius log version} (4), the image of $\xi_{\X}$ by the restriction map is the extension class of the exact sequence
\[
\xymatrix{
0 \ar[r] & B^1_{\X_{s}} \ar[r] & Z^1_{\X_{s}}(\log\,D_s) \ar[r]^-{C_{\X_{s}}} & \Omega_{\X_{s}}(\log\,D_s) \ar[r] & 0
}
\]
for every closed point $s \in S$ by shrinking $S$.
Since $\X_s$ is a globally $F$-split variety over a finite field and the log cotangent bundle is numerical flat, $\xi_{\X_s}=0$ by Theorem \ref{thm:chara toric fib finite feld}.
Therefore, we have $\xi_{\X}=0$, and in particular, $\xi_{X}=0$ by the condition (3).
By Lemma \ref{lem:pre chara}, we obtain the desired result.
\end{proof}

\begin{proof}[Proof of Theorem \ref{intro:thm:toric fibration}]
The implication (2) $\Rightarrow$ (1)  follows from Theorem \ref{thm: chara of toric fib}.
We assume condition (1) in the statement of Theorem \ref{intro:thm:toric fibration}.
Then $\Omega_X(\log\,D)$ is numerically flat.
By Lemma \ref{lem:pre chara} and the proof of \cite{AWZ21}*{Theorem 5.1.1}, $X$ is globally $F$-split.
\end{proof}

\begin{proof}[Proof of Theorem \ref{intro:thm:characterization of abelian varieties}]
The implication (1) $\Rightarrow$ (2) follows from Theorem \ref{intro:thm:toric fibration}.
We assume condition (2).
Then by the proof of Theorem \ref{thm: chara of toric fib}, the exact sequence
\[
0 \longrightarrow B^1_X \longrightarrow Z^1_X \longrightarrow \Omega_X \longrightarrow 0
\]
splits.
By \cite{metha-srinivas}*{Theorem~2}, $X$ is an \'etale quotient of an ordinary abelian variety.
\end{proof}

\begin{proof}[Proof of Theorem~\ref{intro:thm:decomposition}]
Thanks to \cite{KW20}*{Theorem~1.7}, we have a smooth morphism
$\psi:X\to M$ with $\psi_*\mathcal O_X \cong \mathcal O_M$ such that
\begin{itemize}
\item $\psi$ is the MRCC fibration of $X$,
\item every fiber of $\psi$ is a smooth separably rationally connected Fano variety with nef tangent bundle, and 
\item $M$ is a smooth projective variety with numerically flat tangent bundle.
\end{itemize}
Since $X$ is globally $F$-split and $\psi_*\mathcal O_X\cong \mathcal O_M$, we see that $M$ is also globally $F$-split. 
Therefore, by Theorem~\ref{intro:thm:characterization of abelian varieties}, we find an \'etale cover $\pi:A\to M$ from an ordinary abelian variety. 
Put $Y:=X\times_M A$ and let $f:Y\to X$ (resp. $\varphi:Y\to A$) denote the first (resp. second) projection. 
Then $f$ is \'etale, and one can check that $\varphi$ is the MRCC fibration of $Y$. 
Since $X$ is globally $F$-split, so is $Y$.
Then by \cite{Eji19w}*{Proposition~5.11}, we see that $\varphi$ is (locally) $F$-split, so \cite{Eji19w}*{Proposition~5.7} tells us that every fiber of $\varphi$ is globally $F$-split.
\end{proof}

\section{separably rationally connected case}
In this section, we prove Theorem \ref{intro:thm:toric_p}.
First, we recall properties of the residue map.

\begin{prop}\label{prop:exact sequence of log differential}
Let $(X,D)$ be a normal crossing pair over a field $k$.
Then we have the exact sequence
\[
0 \longrightarrow \Omega_X \longrightarrow \Omega_X(\log\,D) \longrightarrow \nu_*\sO_{D^n} \longrightarrow 0,
\]
where $\nu \colon D^n \to D$ is the normalization of $D$.
\end{prop}

\begin{proof}
Let $D=D_1+\cdots+D_r$ be the irreducible decomposition of $D$.
Let $\nu_i \colon D_i^n \to D_i$ be the normalization for every $i$.
We note that $\nu_*\sO_D \simeq \bigoplus_{1 \leq i \leq r} (\nu_i)_*\sO_{D_i^n}$.
If $(X,D)$ is simple normal crossing, then the $\sO_X$-module homomorphism 
\[
\Omega_X(\log\,D) \longrightarrow \bigoplus_{1 \leq i \leq r} \sO_{D_i};\ \phi^{-1}d\phi \mapsto (\mathrm{ord}_{D_i}(\phi))_i
\]
induces the desired exact sequence and the map commutes with \'etale pullbacks.
Therefore, by \'etale descent, we obtain the exact sequence
\[
0 \longrightarrow \Omega_X \longrightarrow \Omega_X(\log\,D) \longrightarrow \nu_*\sO_{D^n} \longrightarrow 0,
\]
as desired.
\end{proof}

\begin{rmk}\label{rmk:exact log diff explicit}
Let $D:=D_1+\cdots+D_r$ be the irreducible decomposition and we assume that there exists $\phi_i \in H^0(X,\sO_X)$ such that $D_i=\mathrm{div}(\phi_i)$.
The map $\Omega_X(\log\,D) \to \nu_*\sO_{D^n}$ in the statement of Proposition~\ref{prop:exact sequence of log differential} is denoted by $\rho$.
By the proof of Proposition~\ref{prop:exact sequence of log differential},
$\rho(\phi_id\phi_i)$ is an element of $\bigoplus_{1 \leq l \leq r} (\nu_l)_*\sO_{D_l^n}$ whose $i$-th component is one and other components are zero, where $\nu_l \colon D_l^n \to D_l$ is the normalization.
\end{rmk}

\begin{prop}\label{prop:unit and differential}
Let $(X,D)$ be a normal crossing pair over a field $k$ and $U:=X \backslash D$.
Let $D:=D_1+\cdots+D_r$ be the irreducible decomposition and $\nu \colon D^n \to D$ the normalization.
\begin{enumerate}
    \item Then we obtain the following commutative diagram of exact sequences;
    \begin{equation} \label{diagram:unit and differential}
    \begin{tikzcd}
    0 \arrow{r} & \sO_X^* \arrow{r} \arrow{d}{\alpha} & \sO_U^* \arrow{r} \arrow{d}{\beta} & \sO_U^*/\sO_X^* \arrow{r} \arrow{d}{\gamma} & 0 \\
    0 \arrow{r} & \Omega_X \arrow{r} & \Omega_X(\log\,D) \arrow{r} & \nu_*\sO_{D^n} \arrow{r} & 0.
    \end{tikzcd}
    \end{equation}
    \item If $H^0(D_i^n,\sO_{D_i^n})=k$ for every $1 \leq i \leq r$, then the map
    \[
    H^0(\sO_U^*/\sO_X^*) \otimes_{\Z} k \longrightarrow H^0(D^n,\sO_{D^n})
    \]
    induced by $\gamma$ is surjective.
\end{enumerate}
\end{prop}

\begin{proof}
First, we construct group maps $\alpha$ and $\beta$.
The map $\alpha$ is defined by $\alpha(\phi)=\phi^{-1}d\phi$, then it is group homomorphism.
Furthermore, if $\phi$ is a local section of $\sO_U^*$, then $\phi^{-1}d\phi$ defines a local section of $\Omega_X(\log\,D)$, thus we can define a group homomorphism $\beta$ by $\beta(\phi)=\phi^{-1}d\phi$.
By construction, we obtain the commutative diagram
\begin{equation*}
\begin{tikzcd}
\sO_X^* \arrow{r} \arrow{d}{\alpha} & \sO_U^* \arrow{d}{\beta} \\
\Omega_X \arrow{r} & \Omega_X(\log\,D).
\end{tikzcd}
\end{equation*}
Therefore, $\beta$ induces the map $\gamma$ fitting the commutative diagram (\ref{diagram:unit and differential}).

Next, we prove the assertion (2).
Let $\nu_i \colon D_i^n \to D_i$ be a normalization for every $1 \leq i \leq r$, then we have $\nu_*\sO_{D^n} \simeq \bigoplus_{1 \leq i \leq r}(\nu_i)_*\sO_{D_i^n}$.
Let $e_i$ be an element of $\bigoplus_{1 \leq l \leq r}H^0(D_l^n,\sO_{D_l^n})$ whose $i$-th component is one and the other components are zero.
Then it is enough to show that the image of $H^0(\gamma)$ contains $e_i$ for all $i$ by assumption.
Let $\{V_j\}$ be an open affine covering of $X$ such that $D_i|_{V_j}=\mathrm{div}(\phi_{ij})$ for some $\phi_{ij} \in \sO_X(V_j)$, then $\phi_{ij} \in \sO_U^*(V_j)$.
The image of $\phi_{ij}$ in $\sO_U^*/\sO_X^*(V_j)$ is denoted by $\overline{\phi_{ij}}$.
By Remark \ref{rmk:exact log diff explicit}, we have $\gamma(\overline{\phi_{ij}})=e_i|_{V_j}$.
Since $\phi_{ij}$ is a generator of $D_i|_{V_j}$, $\{\overline{\phi_{ij}}\}$ defines a global section $\bar{\phi}_i$ of $\sO_X^*/\sO_U^*$.
Therefore, we have $H^0(\gamma)(\bar{\phi}_i)=e_i$, as desired.
\end{proof}

\begin{lem}\label{lem:surj from pic}
Let $X$ be a smooth projective variety over an algebraically closed field $k$ of characteristic $p >0$ and $D$ a normal crossing divisor on $X$.
We consider the group homomorphism
\[
\alpha' \colon \Pic (X) \otimes_{\Z} k \longrightarrow H^1(X,\Omega_X)
\]
induced by 
\[
\sO_X^* \longrightarrow \Omega_X\ ;\ \phi \mapsto \phi^{-1}d\phi.
\]
If $H^1(X,\Omega_X(\log\,D))=0$, then $\alpha'$ is surjective.
\end{lem}

\begin{proof}
By Proposition \ref{prop:unit and differential}, we obtain the commutative diagram
\begin{equation*}
\begin{tikzcd}
    H^0(X,\sO_U^*/\sO_X^*) \arrow{r} \arrow{d}{H^0(\gamma)} & \Pic(X) \arrow{d}{H^1(\alpha)} \\
    H^0(D^n,\sO_{D^n}) \arrow{r} & H^1(X,\Omega_X).
\end{tikzcd}
\end{equation*}
Since $H^1(X,\Omega_X(\log\,D))=0$, the bottom horizontal map is surjective.
By Proposition~\ref{prop:unit and differential}, the map $H^0(\gamma) \otimes_{\Z} k$ is surjective.
Therefore, the map $H^1(\alpha) \otimes_{\Z}k$ is also surjective.
By the construction of $\alpha$, we have $\alpha'=H^1(\alpha) \otimes_{\Z} k$, as desired.
\end{proof}

\begin{thm}\label{thm:chara toric src}
Let $X$ be a smooth projective variety over an algebraically closed field of characteristic $p >0$ and $D$ a normal crossing divisor on $X$.
If $X$ is separably rationally connected and $\Omega_X(\log\,D)$ is numerically flat, then $(X,D)$ is a toric pair.
\end{thm}

\begin{proof}
Since $X$ is separable rationally connected, it is known that $H^0(X,\Omega_X^i)=0$ for $i >0$ and $H^1(X,\sO_X)=0$.
Since $\Omega_X(\log\,D)$ is free by \cite{BPS13}*{Theorem~1.1}, we obtain $H^1(X,\Omega_X(\log\,D))=0$.
Since $B_X^2$ is contained in $F_*\Omega_X^2$, we have $H^0(X,B_X^2)=0$.
We consider the exact sequence
\[
0 \longrightarrow Z_X^1 \longrightarrow F_*\Omega_X^1 \longrightarrow B_X^1 \longrightarrow 0,
\]
then we have the injection
\[
\delta \colon H^1(X,Z^1_X) \hookrightarrow H^1(X,\Omega_X).
\]
Since the image of the map
\[
\alpha \colon \sO_X^* \longrightarrow \Omega_X\ ;\ \phi \mapsto \phi^{-1}d\phi
\]
is contained in $Z^1_X$, we have
\[
\Pic(X) \otimes_{\Z} k \longrightarrow H^1(X,Z_X^1) \overset{\delta}{\longrightarrow} H^1(X,\Omega_X^1), 
\]
where the composition is $\alpha'$ in the statement of Lemma \ref{lem:surj from pic}.
In particular, $\delta$ is an isomorphism by Lemma \ref{lem:surj from pic}.
Thus, we obtain a $p^{-1}$-linear action $\lambda$ on $H^1(X,Z_X^1)$ by
\[
\lambda \colon H^1(X,Z_X^1) \overset{H^1(C)}{\longrightarrow} H^1(X,\Omega_X) \overset{\delta^{-1}}{\longrightarrow} H^1(X,Z_X^1).
\]
\begin{claim}\label{cl;semi-simple}
The cohomology $H^1(X,Z_X^1)$ is generated by $\lambda$-invariant elements.
\end{claim}
\begin{claimproof}
By Lemma \ref{lem:surj from pic} and the above argument, $H^1(X,Z^1_X)$ is generated by the image of the map
\[
H^1(\alpha) \colon \Pic(X) \longrightarrow H^1(X,Z_X^1).
\]
Furthermore, the images of $\alpha$ are $C$-invariant.
Indeed, we have
\[
C(\phi^{-1}d\phi)=\phi^{-p}C(d\phi)=\phi^{-p}\phi^{p-1}d\phi=\phi^{-1}d\phi.
\]
Therefore, we obtain the desired result.
\end{claimproof}

By Claim \ref{cl;semi-simple}, the action $\lambda$ is injective.
In particular, we have the injectivity of $H^1(C)$.
On the other hand, we consider the exact sequence
\[
0 \longrightarrow B^1_X \longrightarrow Z^1_X \overset{C}{\longrightarrow} F_*\Omega_X \longrightarrow 0,
\]
then we have the exact sequence
\[
0 \to H^1(X,B^1_X) \to H^1(X,Z^1_X) \xrightarrow{H^1(C)} H^1(X,\Omega_X^1).
\]
By the injectivity of $H^1(C)$, we have $H^1(X,B^1_X)=0$.
Next, we consider the exact sequence
\[
0 \to B^1_X \to Z^1_X(\log\,D) \to \Omega_X(\log\,D) \to 0.
\]
Since $H^1(X,B_X^1)=0$ and $\Omega_X(\log\,D)$ is free, the above exact sequence splits.
By Lemma \ref{lem:pre chara}, $(X,D)$ has a toric fibration up to \'etale cover.
By \cite{Kollar03}*{Theorem~13}, $(X,D)$ is a toric pair.
\end{proof}

\begin{proof}[Proof of Theorem \ref{intro:thm:toric_p}]
First, we assume that $(X,D)$ is a toric pair.
Since $X$ is rational, it is separably rationally connected.
Furthermore, by Theorem \ref{intro:thm:toric fibration}, $X$ is globally $F$-split and $\Omega_X(\log\,D)$ is numerically flat.
Therefore, we obtain the implications $ (1) \Rightarrow (2)$ and $ (1) \Rightarrow (3)$.
Next, we assume the condition (2).
By Theorem \ref{thm:chara toric src}, $(X,D)$ is a toric pair.
Finally, we assume the condition (3).
By Theorem \ref{intro:thm:toric fibration}, up to \'etale cover, $X$ has a toric fibration over an abelian variety with toric boundary $D$.
By \cite{Kollar03}*{Corollary~13}, the \'etale fundamental group of $X$ is finite, thus $(X,D)$ is toric pair up to \'etale cover.
Therefore, $X$ is separably rationally connected, so we obtain the implication $ (3) \Rightarrow (2)$.
\end{proof}

\begin{proof}[Proof of Theorem \ref{intro:thm:toric_0}]
We assume that $\Omega_X(\log\,D)$ is numerically flat and $X$ is rationally connected.
There exists a smooth affine variety $S$ over $\Spec \Z$, smooth projective variety $\mathcal{X}$ over $S$, reduced divisor $\mathcal{D}$ such that
\begin{enumerate}
    \item $S \to \Spec \Z$ is of finite type and dominant,
    \item $\kappa(\eta) \subseteq k$, where $k$ is the base field of $X$ and $\eta$ is the generic fiber of $S$,
    \item $\mathcal{X}_\eta \times \Spec k \simeq X$ and $\mathcal{D}_\eta \times \Spec k \simeq D$,
    \item $(\mathcal{X},\mathcal{D})$ is normal crossing over $S$.
\end{enumerate}
By \cite{HL10}*{Proposition~3.1}, semistability is an open condition, thus we may assume that $\Omega_{X_{\var{s}}}(\log\,D_{\var{s}})$ is numerically flat for every closed point $s \in S$ by Proposition \ref{prop:numerical flat}, where $\bar{s}$ is the geometric point associated to $s$.
Furthermore, by a similar argument to the proof of \cite{BPS13}*{Theorem~1.1}, we may assume that $X_{\bar{s}}$ is separably rationally connected for every closed point $s \in S$ by shrinking $S$.
By Theorem \ref{thm:chara toric src}, $(X_{\var{s}},D_{\var{s}})$ are toric pairs for all closed points $s \in S$.
By \cite{AWZ21}*{Corollary~4.1.5}, $(X,D)$ is a toric pair.
\end{proof}

\bibliographystyle{skalpha}
\bibliography{bibliography.bib}
\end{document}